\documentclass[a4paper,12pt, psamsfonts, reqno]{amsart}
\usepackage{verbatim}

\usepackage[switch,pagewise,running]{lineno}

\usepackage{amsmath}
\usepackage{amssymb}
\usepackage{textcomp}
\usepackage{amsthm}
\usepackage{amsfonts}
\usepackage[all]{xy}
\usepackage{mathrsfs}
 \usepackage{graphicx}
 \usepackage{epstopdf}
 \usepackage{float}
 \usepackage{graphicx}
 
 \usepackage{epstopdf}
 \usepackage[T1]{fontenc}
 \usepackage{hyperref}
 
\newtheorem{theo}{Th\'eor\'eme}[section]

\def\C{\mathbb{C}}

\def\N{\mathbb{N}}
\def\P{\mathbb{P}}

\def\Z{\mathbb{Z}}

\def\a{\mathbf{a}}
\def\l{\ell}
\def\e{\mathbf{e}}
\def\c{\mathbf{c}}

\def\b{\mathbf{b}}
\def\d{\mathbf{d}}

\def\line{\overline}

\def\Id{\mathop{\mathrm{Id}}\nolimits}
\def\Hom{\mathop{\mathrm{Hom}}\nolimits}

\def\ker{\mathop{\mathrm{ker}}\nolimits}

\def\Im{\mathop{\mathrm{Im}}\nolimits}

\def\dim{\mathop{\mathrm{dim}}\nolimits}

\def\sym{\mathop{\mathrm{sym}}\nolimits}

\def\GL{\mathop{\mathrm{GL}}\nolimits}

\def\remk{\noindent\textit{Remark:~}}

\newtheorem{prop}[theo]{Proposition}
\newtheorem{prop-def}[theo]{Proposition-Definition}
\newtheorem{def-prop}[theo]{Definition-Proposition}

\newtheorem{lemma}[theo]{Lemma}
\newtheorem{teo}[theo]{Theorem}
\newtheorem{definition}[theo]{Definition}

\newtheorem{notation}[theo]{Notation}
\newtheorem{example}[theo]{Example}

\allowdisplaybreaks

\everymath{\displaystyle}

\input xy
\xyoption{all}

\title[Consequences]{Some geometric consequences of the symmetric reduction}

\author{Taiwang DENG}
\address{ 
Beijing Institute of Mathematical Sciences and Applications (BIMSA), Traffic Light Gate, No. 544 Hefangkou Village, Huairou District, Beijing}
\email{dengtaiw@bimsa.cn}
\date{}

\begin{document}

\begin{abstract}
In this article, we derived some consequences to the symmetrization process developed in \cite{Deng23}.
This consists a geometric derivation of  part of the properties which uniquely determines the Kazhdan-Lusztig polynomials of type $A_n$ as
well as an interpretations of the last property by the decomposition theorem of \cite{BBD}. Finally, the relation of geometry of graded nilpotent classes to Parabolic Kazhdan-Lusztig 
polynomials is discussed for the purpose of later applications to Zelevinsky's derivatives.
\end{abstract}

\maketitle
\tableofcontents

\section{Introduction}
This is a continuation of our previous work \cite{Deng23}, which again was part of my thesis\cite{Deng16}.
The goal of the present paper is twofold. Firstly we aim to draw  some consequence from our previous paper  \cite{Deng23}.
In loc.cit., we have developed a symmetrization process to study the geometry of graded nilpotent classes.
In terms of the infinite dimensional representation theory of $\GL_n$ over a non-arhimedean field, this is related to the study
of certain partial derivatives introduced in loc.cit. which generalizes  Zelevinsky's derivatives.

In the first part of the paper, we consider the $q$-polynomial
\[
 P_{x,y}(q): =q^{\frac{1}{2}(\dim(O_{\Phi(y)})-\dim(O_{\Phi(x)}))}\sum_{i}q^{\frac{1}{2}i}\mathcal{H}^{i}(\line{O}_{\Phi(y)})_{\Phi(x)}.
\]
See (\ref{eqn-KZ-geo}) for reference to the relevant notations. The main results in this section are summarized in 
Theorem \ref{prop-geo-KZ}. Perhaps most importantly, in Proposition \ref{prop: 5.3.12} we obtain explicit irreducible 
summands appearing in a special application of the decomposition theorem of \cite{BBD} by comparing with the combinatorial 
properties of $ P_{x,y}(q)$. This will be crucial in our geometric determination of the Zelevinsky's derivatives in \cite{Deng24a}.

In the second part of the paper, we generalize the study of the symmetric case in \cite[\S 4]{Deng23} to the parabolic cases. In the regular cases, the main results are given in Proposition \ref{prop: 6.2.10} and Theorem \ref{teo-main-regular}, which give an explicit parametrization of the poset $S(\a)$ with $\a$ regular by coset representatives in the symmetric group determined by parabolic subgroups of $\GL_n$. The more general case of parametrization for the poset $S(\a)$ with $\a$ non-regular by double coset representatives in the symmetric group is given in Theorem \ref{prop: 6.3.6}.

We will apply the results of this paper to calculate Zelevinsky's derivatives in \cite{Deng24a}.

Finally, the reader may wonder why the author put such an effort into making explicit the relations between the geometry of Schubert varieties and the geometry of graded nilpotent classes. It turns out that there are deeper reasons lying behind the scenes. In fact, it follows from our discussions that a certain category of equivariant perverse sheaves embeds fully faithfully into a certain category of equivariant perverse sheaves on the space of graded nilpotent classes. This gives the dual version of the main results of \cite{Chan23}. A generalization of \cite{Chan23} in our sense to the other groups will be pursued in a forthcoming work.

\par \vskip 1pc
{\bf Acknowledgements}
This paper is part of my thesis at University Paris 13, which is funded by the program DIM of the
region Ile de France. I would like to thank my advisor Pascal Boyer for his keen interest in this work and 
his continuing support and countless advice.
It was rewritten during my stay as a postdoc at Max Planck Institute for Mathematics and Yau Mathematical Sciences Center, I thank their hospitality. In addition, 
I would like to thank Alberto M\'inguez ,Vincent S\'echerre, Yichao Tian and Bin Xu for their helpful discussions on the subject.

\section{Geometric Proof of KL Relations}

For $n\geq 1$, recall that the permutation group $S_n$ of $\{1, \cdots, n\}$
and that $S=\{\sigma_i=(i, i+1): i=1, \cdots, n-1\}$ is a set of generators. On $S_n$ there is a partial order "$<$" called the Bruhat order (cf. \cite[Chapter 2]{BF}).

We will keep the same notations as in \cite[\S 2]{Deng23}. 
Recall that in \cite[Definition 2.12]{Deng23} we have introduced a partial order $\leq$ on the set of multisegments and the finite set
\[
S(\a)=\{\b: \b\leq \a\}
\]
following \cite[7.1]{Z2}. 
Following \cite[Proposition 4.7]{Deng23}, 
 we can find a multisegment $\a_{\Id}$ such that we have 
an isomorphism
\begin{equation}\label{eqn-parametrization-sym}
 \Phi: S_n\rightarrow S(\a_{\Id}).
\end{equation}
Recall that according to \cite[Definition 3.1]{Deng23}, for each multisegment we can associate a weight function 
\[
\varphi_{\a}: \Z\rightarrow \N.
\]
Furthermore, by \cite[Definition 3.2]{Deng23} we can attach a vector space $E_{\varphi}$ to a weight function $\varphi$ together with
an action by an algebraic group $G_{\varphi}$. By \cite[Proposition 3.3]{Deng23} the $G_{\varphi}$-orbits on $E_{\varphi}$ are indexed by the set $S(\varphi)$ of multisegments with weight $\varphi$. Now for $x, y\in S_{n}$, we introduce the polynomial in $q$ (cf. \cite[Definition 3.5]{Deng23}):
\begin{equation}\label{eqn-KZ-geo}
 P_{x,y}(q): =q^{\frac{1}{2}(\dim(O_{\Phi(y)})-\dim(O_{\Phi(x)}))}\sum_{i}q^{\frac{1}{2}i}\mathcal{H}^{i}(\line{O}_{\Phi(y)})_{\Phi(x)}.
\end{equation}

The goal in this section is:

\begin{teo}\label{prop-geo-KZ}
The following relations hods for $P_{x, y}(q)$:
 \begin{description}
  \item[(1)] $P_{x,x}=1$ for all $x\in S_{n}$;
  \item[(2)]if $x<y$ and $s\in S$,  are such that $sy<y$, $sx>x$,
  then $P_{x,y}=P_{sx, y}$;
   \item[(3)]if $x<y$ and $s\in S$, are such that $ys<y$, $xs>x$,
  then $P_{x,y}=P_{xs, y}$;
   \item[(4)]if $x<y$ and $s\in S$, are such that $sy<y$, $sx<x$,
   and $x$ is not comparable to $sy$, 
  then $P_{x,y}=P_{sx, sy}$;
  \end{description}
  Moreover, a geometric interpretation of the following relation is given in Proposition \ref{prop: 5.3.12}: 
  \begin{description}
   \item[(5)]if $x<y$ and $s\in S$, are such that $sy<y$, $sx<x$, and $x<sy$,
  then
  \[
   P_{x,y}=P_{sx, sy}+qP_{x, sy}-\sum_{x\leq z <sy,sz<z}q^{1/2(\ell(y)-\ell(z))}\mu(z, sy)P_{x,z},
  \]
here $\mu(z, sy)$ is the coefficient of degree $1/2(\ell(sy)-\ell(z)-1)$ in $P_{z, sy}$ defined to be zero if 
$\ell(sy)-\ell(z)$ is even).
 \end{description}
\end{teo}
Our proof of the relations (1)-(4) are geometric and relies on the geometry study of \cite[\S 5.3]{Deng23}.
We give an interpretation geometric for (5)
which will be used in our work on the geometric study of Zelevinsky's derivatives \cite{Deng24a}.
Note that it follows from \cite{KL79} that the above properties (1)-(5) characterize
a unique family of polynomials $P_{x,y}(q)$ of $\Z[q]$ for $x, y\in S_n$ which are the Kazhdan-Lusztig polynomials.

\subsection{Relation (2) and (3)}
Assume that $k, k_1\in \N$ such that $1<k_1\leq n, k=n+k_1-1$.
Since the relation (2) and (3) are symmetric to each other, 
we only prove (2). By \cite{BF} (1.26), 
the conditions 
\[
 \sigma_{k_1-1}w>w, ~\sigma_{k_1-1}v<v.
\]
are equivalent to 
\[
 w^{-1}(k_1-1)<w^{-1}(k_1), \quad v^{-1}(k_1-1)>v^{-1}(k_1).
\]

\begin{prop}
Let $\a=\Phi(w), \c=\Phi(v)\in S(\a)$, such that 
\[
 w^{-1}(k_1-1)<w^{-1}(k_1), \quad v^{-1}(k_1-1)>v^{-1}(k_1),
\]
then 
\[
 P_{w, v}(q)=P_{\sigma_{k_1-1}w, v}(q).
\]
\end{prop}

\begin{proof}

Suppose that 
\[
 \Phi(\Id)=\{\Delta_1\preceq \cdots \preceq \Delta_n\}.
\]
Let $\b=\Phi(\sigma_{k_1-1}w)$, then 
\begin{align*}
 \b=&\sum_{j}[b(\Delta_j), e(\Delta_{\sigma_{k_1-1}w(j)})]\\
   =&\sum_{j}[b(\Delta_{w^{-1}\sigma_{k_1-1}(j)}), e(\Delta_j)]\\
   =&\sum_{j\neq k_1-1, k_1}[b(\Delta_{w^{-1}(j)}), e(\Delta_j)]+
   [b(\Delta_{w^{-1}(k_1-1)}), e(\Delta_{k_1})]+[b(\Delta_{w^{-1}(k_1)}), e(\Delta_{k_1-1})], 
\end{align*}
where $b(\Delta)$ (resp. $e(\Delta)$ )denotes the beginning (resp. the end) of a segment $\Delta$.
Note that 
\[
 e(\Delta_{k_1-1})=n+k_1-2=k-1, ~e(\Delta_{k_1})=n+k_1-1=k,
\]

then $\b^{(k)}=\a^{(k)}$ (cf. \cite[Definition 5.2]{Deng23} for the definition of $\a^{(k)}$).
Now applying \cite[Corollary 5.38]{Deng23}
gives the result.
\end{proof}

\subsection{Relation (4)}

Let $\a=\Phi(\Id)$ (cf.(\ref{eqn-parametrization-sym})) and $ \varphi=\varphi_{\a}$ be the corresponding weight function.
We first recall some results from \cite[\S 5.3]{Deng23}. The variety $X_{\a}^{k}$ is defined in \cite[Definition 5.20]{Deng23}, which is a suitable sub-variety of $E_{\varphi}$ admitting a fibration $\alpha$ over 
the Grassmanian $Gr(\ell_{\a,k}, V_{\varphi})$ (cf. \cite[Proposition 5.28]{Deng23}) and
\[
\ell_{\a, k}=\sharp\{\Delta\in \a: e(\Delta)=k\}, (\text{ cf. \cite[Notation 5.12]{Deng23}}).
\]
For fixed $W\in Gr(\ell_{\a,k}, V_{\varphi})$, 
let $(X_{\a}^{k})_{W}$ be the fiber of $\alpha$ over $W$ (cf. \cite[Notation 5.30]{Deng23}).
By 
\cite[Proposition 5.35]{Deng23}, 
we have an open immersion
\[
 \tau_{W}:  (X_{\a}^{k})_{W}\rightarrow 
 (Z^{k, \a})_{W}\times \Hom(V_{\varphi, k-1}, W).
\]
Here $Z^{k, \a}$ is a sub-variety of a certain fiber bundles $\tilde{Z}^k$ over the Grassmanian $Gr(\ell_{\a,k}, V_{\varphi})$ (cf. \cite[Proposition 5.24]{Deng23}) and $(Z^{k, \a})_W$ is the fiber over $W$.
See also \cite[Definition 3.2]{Deng23}
for  the definition of $V_{\varphi, k-1}$.
\begin{definition}
By composing with the canonical projection $$(Z^{k, \a})_{W}\times \Hom(V_{\varphi, k-1}, W)
\rightarrow (Z^{k, \a})_{W},$$ we have a morphism
\[
 \phi_W:  (X_{\a}^{k})_{W}\rightarrow (Z^{k, \a})_{W}.
\]
\end{definition}

Recall that for a fixed weight function $\varphi$ and $\a\in S(\varphi)$,  we denote by $O_{\a}$ the $G_{\varphi}$-orbit indexed by $\a$
on $E_{\varphi}$ (cf. \cite[Proposition 3.3]{Deng23}).
\begin{prop}\label{prop: 4.4.5}
For any $ \b=\Phi(w)\in S(\a)_{k}$ (cf. \cite[Definition 5.9]{Deng23}), we have 
 \[
  \psi_k^{-1}(\b^{(k)})=\{\b, \b'=\Phi(\sigma_{k_1-1}w)\}, 
 \]
 where $\psi_k$ is defined in \cite[Definition 5.6]{Deng23}.
Moreover, $\phi_W$ is a fibration such that we have 
\begin{description}
\item[(1)]an isomorphism $\phi_W^{-1}(T^{(k)})\simeq (\C^2-\{0\})\times \C^{2n-k-1}$ for $T^{(k)}\in O_{\b^{(k)}}$,
\item[(2)]and $\phi_W^{-1}(T^{(k)})\cap O_{\b'}\simeq \C^{\times}\times \C^{2n-k-1}$ for $T^{(k)}\in O_{\b^{(k)}}$.

\end{description}
\end{prop}

\begin{proof}
Note that we have 
\[
 \psi_k^{-1}(\b^{(k)})\subseteq S(\b^{(k)}+[k]),
\]
we observe that 
\[
 S(\b^{(k)}+[k])\cap S(\a)=S(\b'),
\]
Since $\b$ is minimal in $\psi_k^{-1}(\b^{(k)})$ 
by \cite[Proposition 5.36]{Deng23}, we have 
\[
 \psi_k^{-1}(\b^{(k)})=\{\b, \b'\}.
\]

 Then consider the restriction to the fiber over $W$:
\[
 \phi_W: (O_{\b}\cup O_{\b'})_{W}\rightarrow O_{\b^{(k)}}.
\]
Let $T\in O_{\b}\cup O_{\b'}, ~T_{0}\in \Hom(V_{\varphi, k-1}, W)$.
Define $T'\in E_{\varphi}$ by
\[
 T'|_{V_{\varphi, k-1}}=T_{0}\oplus T^{(k)}|_{V_{\varphi, k-1}},
 \]
\[
 T'|_{V_{\varphi, k}}=T^{(k)}|_{V_{\varphi, k}/W}\circ p_{W},
\]
\[
 T'|_{V_{\varphi, i}}=T^{(k)}, \text{ for } i\neq k-1, k.
\]
 The map $T^{(k)}$ is defined in \cite[Definition 5.2]{Deng23}. We know that in the case of consideration $\dim(W)=\ell_{k}=1$, and 
for $\dim(\ker(T^{(k)}|_{V_{\varphi, k-1}}))=2$. 
Now let 
\[
 \Delta_{1}<\Delta_{2}
\]
be the two segments in $\b^{(k)}$ which ends in 
$k-1$.
And we consider the following flag
\[
 V_{0}=\ker(T^{(k)}|_{V_{\varphi,k-1}})\supseteq 
 V_{1}=\Im(T^{(k)})^{\Delta_{2}}\cap \ker(T^{(k)}|_{V_{\varphi,k-1}}).
\]
And we have $\dim(V_{1})=1$. 
Then for $T'\in O_{\b}\cup O_{\b'}$, it is necessary and sufficient that 
\[
T_{0}(V_{0})\neq 0.
\]
This amounts to give a nonzero element in $\Hom(V_{0}, W)\simeq \C^2$, 
which proves that the fiber  $\phi_W^{-1}(T^{(k)})\simeq (\C^{2}-{0})\times \C^{2n-k-1}$,
where the factor $\C^{2n-k-1}$ comes from the fact that $\dim(V_{\varphi, k-1})=2n-(k-1)=2n-k+1$. 
As for $T'\in O_{\b'}$, it is necessary and sufficient that 
\[
  T_{0}(V_{1})=0, ~T_{0}(V_{0})\neq 0,
\]
which amounts to give a zero element in $\Hom(V_{0}/V_{1}, W)\simeq \C$.
Hence $\phi_W^{-1}(T^{(k)})\cap O_{\b'}\simeq\C^{\times}\times \C^{2n-k-1}$. To see that 
$\phi_W$ is a fibration,
 fix $V\subseteq V_{\varphi, k-1}$ such that 
 $\dim(V)=2$. Consider the sub-scheme of $(Z^{k, \a})_{W}$ given by 
\[
  U_{V}=\{T\in (Z^{k, \a})_{W}: \ker(T|_{V_{\varphi, k-1}})=V\}.
\]
Note that since $\dim(V_{\varphi, k-1})=\dim(V_{\varphi, k}/W)+2$, the 
fact that $$\dim(\ker(T|_{V_{\varphi,k-1}}))=2$$ implies that $U_{V}$
is actually open in $(Z^{k, \a})_{W}$. 
In this case 
\[
 \phi_W^{-1}(U_{V})=U_{V}\times (\Hom(V, W)-\{0\})\times \Hom(V_{\varphi, k-1}/V, W).
\]

\end{proof}

\begin{prop}\label{prop-geo-relation(4)}
 Let $\b=\Phi(w), \c=\Phi(v)\in S(\a)$, such that 
\[
 w^{-1}(k_1-1)>w^{-1}(k_1), \quad v^{-1}(k_1-1)>v^{-1}(k_1), \quad w<v, 
\]
and $w$ is not comparable with $\sigma_{k_1-1}v$,
then 
\[
 P_{w, v}(q)=P_{\sigma_{k_1-1}w, \sigma_{k_1-1}v}(q).
\]
\end{prop}

\remk As before, our conditions are equivalent to 
\[
 \sigma_{k_1-1}w>w, ~ \sigma_{k_1-1}v>v.
\]

\begin{proof}
Note that our assumption implies that both $\b$ and $\c$ are in 
$S(\a)_{k}$. 
Let $b'=\Phi(\sigma_{k_1-1}w), \c'=\Phi(\sigma_{k_1-1}v)
$. Then $\b'>\c'$.

For $\b>\d>\c$, we must have $\d=\Phi(\alpha)$ with $\sigma_{k_1-1}\alpha<\alpha$.
In fact,  $\sigma_{k_1-1}\alpha>\alpha$ would imply $\d>\c'$ by 
lifting property of Bruhat order (cf. \cite{BF} proposition 2.2.7). 
Now that we have $\b>\d>\c'$, contradicting to our assumption that $\b$
is not comparable to $\c'$.
Let $\d'=\Phi(\sigma_{k_1-1}\alpha)$. Note that we construct in this way 
 a map between the sets 
$$\rho: \{\d: \b\geq \d\geq \c\}\rightarrow \{\d': \b'\geq \d'\geq \c'\}$$
sending $\d$ to $\d'$. We first establish two lemmas before we continue.
\begin{lemma}
The morphism $\rho$ is a bijection.  
\end{lemma}

\begin{proof}
Let $\e'=\Phi(\beta)\in S(\b')$ with $\e'>\c'$. 
We show that $\sigma_{k_1-1}\beta>\beta$.
In fact, assume that $\sigma_{k_1-1}\beta<\beta$.
Then the lifting property of Bruhat order implies 
$\b>\e'>\c'$, which is a contradiction to the fact that 
$\b$ is not comparable to $\c'$.
Hence we have $\e=\Phi(\sigma_{k_1-1}\beta)<\e'$.
Moreover, since $\sigma_{k_1-1}w<\beta<\sigma_{k_1-1}v$,
and $w>\sigma_{k_1-1}w,~ v>\sigma_{k_1-1}v$, we have 
\[
 w<\sigma_{k_1-1}\beta<v, 
\]
hence $\b>\e>\c$. This proves the surjectivity.
The injectivity is clear from the definition.

\end{proof}

\begin{lemma}
 The restricted morphism 
 \[
  \phi_W: X_{\b', \c'}^{k}\rightarrow Z_{b^{(k)}, \c^{(k)}}^k(cf. \text{\cite[Definition 5.25]{Deng23}})
 \]
is a fibration with fibers isomorphic to $\C^{\times}\times \C^{n-k}$.
\end{lemma}

\begin{proof}
Since $\phi_W$ is a composition of 
$\tau_W$, which is an open immersion,  and a canonical
projection, to show that it is a fibration, it suffices
to show that all of its fibers are isomorphic to $\C^{\times}\times \C^{n-k}$.
This follows from Proposition \ref{prop: 4.4.5}
and the fact that for any $\d'\in S(\b')$ we have $\d'\notin S(\a)_{k}$.
\end{proof}
Let us continue the proof of Proposition \ref{prop-geo-relation(4)}.
From the above lemmas we obtain
\[
 P_{\b', \c'}(q)=P_{\b^{(k)}, \c^{(k)}}(q).
\]
Now we are done by applying \cite[Corollary 5.38]{Deng23},
i.e,
\[
 P_{\b^{(k)}, \c^{(k)}}(q)=P_{\b, \c}(q).
\]
Hence we are done.
\end{proof}

\subsection{Relation (5)}

Finally, we arrive at the relation (5).
We will give an interpretation of this relation in terms of 
the decomposition theorem (See \cite{BBD}).

We use $V^*$ to denote the dual vector space of $V$.
\begin{definition}
Let 
\begin{align*}
 \mathfrak{Z}_{W}^k=&\{(T,z)\in (Z^{k, \a})_{W}\times 
 \Hom(V_{\varphi, k-1}^*, W^*):  \text{ and }z \\
 &\text{ factors through the canonical projection  } 
 V_{\varphi, k-1}^*\rightarrow \ker(T|_{V_{\varphi, k-1}})^*\}.
\end{align*}
\end{definition}

\begin{prop}\label{prop: 4.4.10}
 The canonical projection $\mathfrak{Z}_{W}^k\rightarrow (Z^{k, \a})_{W}$
 turns $\mathfrak{Z}_{W}^k$ into a vector bundle of rank 2 over $(Z^{k, \a})_{W}$. 
\end{prop}
 
 \begin{proof}
 Note that we have $\dim(\ker(T|_{V_{\varphi, k-1}}))=2$ and 
 $\dim(W)=1$. 
 Note that by taking dual,  as a scheme,  $\mathfrak{Z}_{W}^k$ is isomorphic 
 to the scheme parametrize the data $(T, z)\in (Z^{k, \a})_{W}\times V_{\varphi, k-1}$
 such that $z\in \ker(T|_{V_{\varphi, k-1}})$.
 Fix $V\subseteq V_{\varphi, k-1}$ such that 
 $\dim(V)=2$. Consider the sub-scheme of $(Z^{k, \a})_{W}$ given by 
\[
  U_{V}=\{T\in (Z^{k, \a})_{W}: \ker(T|_{V_{\varphi, k-1}})=V\}.
\]
As is showed in Proposition \ref{prop: 4.4.5},  $U_{V}$
is actually open in $(Z^{k, \a})_{W}$. 
Using the previous dual description of $\mathfrak{Z}_{W}^k$, we observe that 
the open set $U_{V}$ trivializes the projection $\mathfrak{Z}_{W}^k\rightarrow (Z^{k, \a})_{W}$.

\end{proof}

\begin{definition}
Let $\mathcal{Z}_{W}^k= Proj_{Z^{k, \a}_W}(\mathfrak{Z}_W^k)$
be the projectivization of the vector bundle $\mathfrak{Z}_{W}^k\rightarrow Z_{W}^k$.
And we shall denote the structure morphism by $\kappa_{W}^k: \mathcal{Z}_{W}^k\rightarrow Z_{W}^k $.
\end{definition}

\begin{definition}
From now on, we fix a pair of non-degenerate bi-linear forms
\[
 \zeta_{k-1}: V_{\varphi, k-1}\times V_{\varphi, k-1} \rightarrow \C, 
 ~\zeta_{k}: V_{\varphi, k}\times V_{\varphi, k} \rightarrow \C.
\]
which allows us to have an identification $\eta_{i}: V_{\varphi, i}\simeq V_{\varphi, i}^*$,
for $i=k-1, k$. 
\end{definition}

\begin{definition}
Let $T\in (X_{\a}^k)_{W}$, then we define 
\[
 \lambda: (X_{\a}^k)_{W}\rightarrow (X_{\a}^k)_{\eta_k(W)},
\]
by letting
\[
 \lambda(T)|_{V_{\varphi, k-2}}=\eta_{k-1}\circ T|_{\varphi, k-2} 
\]

\[
 \lambda(T)|_{V_{\varphi, k-1}}=\eta_k\circ T|_{\varphi, k-1}\circ \eta_{k-1}^{-1},
\]
\[
 \lambda(T)_{V_{\varphi, k}}=T|_{\varphi, k}\circ \eta_k^{-1},
\]
and 
\[
 \lambda(T)_{V_{\varphi, i}}=T|_{\varphi, i}, \text{ for } i\neq k-2, k-1, k.
\]
\end{definition}

\begin{lemma}
We have 
$\ker(\lambda(T)|_{V_{\varphi, k}})=\eta_k(W)$,
and $$\ker(\lambda(T)^{(k)}|_{V_{\varphi, k-1}})=\eta_{k-1}(\ker(T|_{V_{\varphi,k-1}})).$$ 
\end{lemma}

\begin{proof}
The fact  $\ker(\lambda(T)|_{V_{\varphi, k}})=\eta_k(W)$ follows from definition.
Note that 
\[
 \ker(T^{(k)}|_{V_{\varphi, k-1}})=\{v\in V_{\varphi, k-1}:  T(v)\in W\}=T|_{V_{\varphi, k-1}}^{-1}(W).
\]
Since
\[
 (\lambda(T)|_{V_{\varphi, k-1}})^{-1}(\eta_k(W))=\eta_{k-1}(T|_{V_{\varphi, k-1}})^{-1}(W)
 =\eta_{k-1}(\ker(T|_{V_{\varphi,k-1}})),
\]
hence 
\[
  \ker (\lambda(T)^{(k)}|_{V_{\varphi, k-1}})=\eta_{k-1}(\ker(T|_{V_{\varphi,k-1}})).
\]

\end{proof}

\begin{definition}
We define
\[
\xi_W: (X_{\a}^k)_W\rightarrow \mathcal{Z}_{W}^k,
\]
for $T\in (X_{\a}^k)_W$, then 
\[
 \xi_W(T)=(T^{(k)}, \lambda(T)|_{\ker(\lambda(T)^{(k)}|_{V_{\varphi, k-1}})}).
\]
 \end{definition}
This is well defined since  
\[
 \lambda(T)|_{\ker(\lambda(T)^{(k)}|_{V_{\varphi, k-1}})}\in\Hom(\ker(\lambda(T)^{(k)}|_{V_{\varphi, k-1}}),\eta_k(W) ),
\]
and
$$
 \Hom(\ker(\lambda(T)^{(k)}|_{V_{\varphi, k-1}}),\eta_k(W) )
\simeq \Hom(\ker(T^{(k)}|_{V_{\varphi, k-1}})^*, W^*),
$$ 
and $\lambda(T)|_{\ker(\lambda(T)^{(k)}|_{V_{\varphi, k-1}})}\neq 0$.

\begin{prop}
 The morphism $\xi_{W}$ is a fibration with fibers isomorphic to $\C^{\times}\times \C^{n-k}$.
\end{prop}

\begin{proof}
Let $V\subseteq V_{\varphi, k-1}$ be a subspace such that $\dim(V)=2$. 
Consider the open sub-schemes, 
\[
 U_{1,V}=\{(T, z)\in \mathfrak{Z}_{W}^k: z\neq 0,~ \ker(T|_{V_{\varphi, k-1}})=\eta_{k-1}^{-1}(V)\}, 
 \]
 \[
 U_{V}=\{T\in (Z^{k, \a})_{W}, ~\ker(T|_{V_{\varphi, k-1}})=\eta_{k-1}^{-1}(V)\}.
 \]
Let $\tilde{U}_{1, V}$ be the image of $U_{1, V}$
in $\mathcal{Z}_{W}^k$ by the canonical projection.
As indicated in the proof of  Proposition \ref{prop: 4.4.10}, 
the set $U_{V}$ trivialize the morphism $\mathfrak{Z}_{W}^k\rightarrow (Z^{k, \a})_{W}$,
hence 
\[
 \tilde{U}_{1, V}\simeq U_{V}\times (\Hom(V, \eta_{k}(W))-\{0\}/\C^{\times})
\]

Note that we have 
\[
 \Hom(V, \eta_{k}(W))\simeq \Hom(\eta_{k-1}^{-1}(V), W).
\]
And by \cite[Proposition 5.35]{Deng23} and Proposition \ref{prop: 4.4.5}, 
we have the following isomorphism
\[
 \xi_{W}^{-1}(\tilde{U}_{1, V})\simeq U_{V}\times (\Hom(\eta_{k-1}^{-1}(V), W)-0)\times 
 \Hom(V_{\varphi, k-1}/\eta_{k-1}^{-1}(V), W).
\]
Hence for any $(T, z)\in \mathcal{Z}_{W}^k$ such that 
$\ker(T|_{V_{\varphi, k-1}})=\eta_{k-1}^{-1}(V)$
, let $U_{2, V}$ be an open 
subset of $(\Hom(\eta_{k-1}^{-1}(V), W)-0)/\C$ which trivializes the bundle
\[
 (\Hom(\eta_{k-1}^{-1}(V), W)-0)\rightarrow (\Hom(\eta_{k-1}^{-1}(V), W)-0)/\C,
\]
then the open sub-scheme $U_{V}\times U_{2, V}$ of $\tilde{U}_{V, 1}$ trivialize the 
morphism $\phi_W$ as a neighborhood of $(T, z)$.
 
\end{proof}

\begin{definition}
Let $\b>\c$ be two elements in $\tilde{S}(\a)_{k}$, then we define
\[
 \mathcal{Z}^k_{\b, \c}=\xi_{W}((X^k_{\b, \c})_W).
\]
And 
\[
 \mathcal{Z}^k(\b)=\xi_W((O_{\b})_W).
\]
\end{definition}

\begin{definition}
 Let $w<v$ be two elements in $S_{n}$ such that $\sigma_{k_1-1}v<v$. We define 
 \[
 R(w, v)_{k_1}=\{z: w\leq z<\sigma_{k_1-1} v, \sigma_{k_1-1}z<z\}.
 \]
And we denote $R(\Id, v)_{k_1}$ by $R(v)_{k_1}$. 
\end{definition}
 
Now let $\b=\Phi(w)$, $\c=\Phi(v)$ such that 
\[
 w(k_1-1)>w(k_1),~ v(k_1-1)>v(k_1).
\]
And let $\b'=\Phi(\sigma_{k_1-1}w), \c'=\Phi(\sigma_{k_1-1}v)$. We assume that 
\[
 \b>\c, ~\b>\c',
\]
which coincide with the assumption in relation (5) at the beginning of this chapter.

Now we apply the decomposition theorem to the projective morphism
\[
 \kappa_W: \mathcal{Z}^k_{\b', \c'}\rightarrow Z^{k, \a}_{\b^{(k)}, \c^{(k)}, W}.
\]
which asserts that there exists a finite collection of triples $(\d_{i}, L_{i}, h_{i}: i=1, \cdots, r)$, 
with $\d\in S(\a)_{k}$, ~$\b^{(k)}\leq \d_{i}^{(k)}<\c^{(k)}$, where $L_{i}$ is a vector spaces over $\C$, 
such that 
\addtocounter{theo}{1}
\begin{equation}\label{eq: (4)}
 R(\kappa_W)_{*}IC(\mathcal{Z}_{\b', \c'}^k)=IC(Z^{k, \a}_{\b^{(k)}, \c^{(k)}, W})\oplus_{i=1}^{r}IC(Z^{k, \a}_{\b^{(k)}, \d^{(k)}_{i}, W}, L_{i})[h_{i}].
\end{equation}

Now localize at a point $x_{\b^{(k)}}\in O_{\b^{(k)}}$, we have know that 
the Poincar\'e series of $(IC(Z^{k, \a}_{\b^{(k)}, \c^{(k)},W}))_{x_{\b^{(k)}}}$ is given by 
$P_{\b, \c}(q)=P_{w, v}(q)$. 

\begin{lemma}
The Poincar\'e series of $R\Gamma(\kappa_W^{-1}(x_{\b^{(k)}}), IC(\mathcal{Z}_{\b', \c'}^k))$ 
is given by $P_{\sigma_{k_1-1}w, \sigma_{k_1-1}v}(q)+qP_{w, \sigma_{k_1-1}v}(q)$,
where $\Gamma$ is the functor of taking global sections.
\end{lemma}

\begin{proof}
 Note that by assumption, we have 
 \[
  \kappa_W^{-1}(x_{\b^{(k)}})\simeq \P^1
 \]
such that $ \kappa_W^{-1}(x_{\b^{(k)}})\cap \mathcal{Z}^k(\b')=\{pt\}$ and 
$\kappa_W^{-1}(x_{\b^{(k)}})\cap \mathcal{Z}^k(\b)\simeq \P^1-\{pt\}$.
And we have the following exact sequence 
\[
0\rightarrow  IC(\mathcal{Z}_{\b', \c'}^k)|_{pt}\rightarrow IC(\mathcal{Z}_{\b', \c'}^k)\rightarrow 
 IC(\mathcal{Z}_{\b', \c'}^k)|_{\kappa_W^{-1}(x_{\b^{(k)}})-\{pt\}}\rightarrow 0.
\]
Taking the Poincar\'e series gives the result.

\end{proof}

Now it is clear that  (\ref{eq: (4)})
will give rise to an equation of the form as that in relation (5) in Theorem \ref{prop-geo-KZ}.
Comparing the two equations, we get 

\begin{prop}\label{prop: 5.3.12}
The collection of triples $(\d_{i}, L_{i}, h_{i}: i=1, \cdots, r)$ are given by 
\begin{description}
 \item [(1)]We have $\{\d_{i}: i=1, \cdots, r\}=\{z\in R(w, v)_{k}: \mu(z, \sigma_{k-1}v)\neq 0\}$.
 \item [(2)] If $\d_{i}=\Phi(z)$, then $L_{i}\simeq \C^{\mu(z, \sigma_{k-1}v)}$.
 \item [(3)] If $\d_{i}=\Phi(z)$, then $h_i=\l(v)-\l(z)$.
\end{description} 
\end{prop}

\begin{proof}
Note that 
the Poincar\'e series of the intersection complex 
 $$IC(Z^{k, \a}_{\b^{(k)}, \d^{(k)}_{i}, W}, L_{i})[h_{i}]$$
is 
\[
 \dim(L_i)q^{1/2h_i}P_{w, z_i}(q),
\]
where $\d_i=\Phi(z_i)$. 
Now compare the polynomials given by (\ref{eq: (4)})
and the relation (5) in Theorem \ref{prop-geo-KZ}, we get our results.
\end{proof}

\remk Note that one should be able to deduce the above results from a general statement about the 
decomposition theorem. We leave this for future work. 

\remk It seems that we have done here 
may be generalized to give the normality of for general $\line{O}_{\b}$
instead of using the results of Zelevinsky.

\section{Geometry related to the poset \texorpdfstring{$S(\a) $}{Lg}}

Let $\a$ be a multisegment and $S(\a)=\{\b\leq \a\}$ the associated poset defined in  
\cite[Definition 2.12]{Deng23}. The aim of this section is to identify the poset 
structure of $S(\a)$ in a way that is analogues to the symmetric case (\ref{eqn-parametrization-sym}).

In the first subsection we consider the case where 
$\a$ is regular (cf. \cite[Definition 4.2]{Deng23}) and prove that $S(\a)$  is an 
interval in $S_m\simeq B\backslash GL_m/B$, where 
$m$ is the number of segments in $\a$ and $B$ is the Borel subgroup.

In the general case we identify $S(\a)$ with an interval 
in a parabolic quotient $S_{J_1}\backslash S_m/S_{J_2}$ of $S_m$ 
given in section 2 related to the double quotient $P_{J_1}\backslash GL_m/ P_{J_2}$,
where $P_{J_1}$ and $P_{J_2}$ are parabolic subgroups.

\subsection{Regular Case}

Our goal in this section is to prove that 
for general regular multisegment $\a$, 
the set $S(\a)$ is isomorphic to some Bruhat interval $[x, y]$ for 
$x, y\in S_{n}$, where $n$ depends on $\a$.

Recall that we have introduced a hypothesis $H_k(\a)$ in \cite[Definition 5.3]{Deng23} and use it
to define $S(\a)_k$ in \cite[Definition 5.9]{Deng23}.

\begin{lemma}
Assume that $\b\in S(\b)_{k}$ such that $\b$ and $\b^{(k)}$ (cf. \cite[Definition 5.2]{Deng23}) are both regular.
Let $\c\in S(\b)_k$. Then for $\d\in S(\b)$ and $\d>\c$, we have $\d\in S(\b)_k$. 
\end{lemma}

\begin{proof}
It suffices to show that $\d$ satisfies the hypothesis $H_{k}(\b)$.
Note that $e(\d)=\{e(\Delta): \Delta\in \d\}$ is a set because $\d$ is regular 
and as is indicated in the proof of \cite[Proposition 4.4]{Deng23} we have $e(\d)\subseteq e(\b)$ .

Note that  $k-1\notin e(\b)$ since $\b\in S(\b)_k$ and $\b^{(k)}$ is regular
, therefore it is not in $e(\d)$ either. Hence to show that $\d\in S(\b)_k$
 hence it is equivalent to show that $k\in e(\d)$.
Since $\c\in S(\d)$, we know that $e(\c)\subseteq e(\d)$ .
Now that $k\in e(\c)$, we conclude that $k\in e(\d)$. We are done. 
\end{proof}

\begin{lemma}\label{lem: 6.1.3}
Now let $\b'\in S(\b)_k$ such that $\psi_k(\b')=(\b^{(k)})_{\min}$ which is the minimal element in $S(\b^{(k)})$ with respect to the partial order $\leq$. 
We have 
\[
 S(\b)_k=\{\c\in S(\b): \c\geq\b'\}.
\] 
\end{lemma}

\begin{proof}
 By the lemma above, we know that 
 $S(\b)_k\supseteq \{\c\in S(\b): \c\geq\b'\}$.
 We conclude that we have equality since $\psi_k$ preserve the order.
\end{proof}

\begin{prop}
Assume that $\a$ is regular. Then 
\begin{description}
\item[(1)] 
There exists a symmetric multisegment $\a^{\sym}$ such that 
\[
 S(\a)\simeq S(\a^{\sym})_{k_{r}, \cdots, k_{1}}.
\]

\item[(2)] There exists an element $\a'\in S(\a^{\sym})$ such that 
\[
 S(\a^{\sym})_{k_{r}, \cdots, k_{1}}=\{\c\in S(\a^{\sym}): \c\geq \a'\}.
\]

\end{description}
\end{prop}

\begin{proof}
Note that $(1)$ follows directly from the analogues for \cite[Proposition 6.10]{Deng23} 
and $(2)$ follows from applying successively Lemma \ref{lem: 6.1.3} to 
the sequence obtained in 
the lemma below.  
\end{proof}

\begin{lemma}
There exists a sequence of multisegments  $\a_{0}=\a, \cdots,\a_{r}=\a^{\sym}$ such that 
$\a^{\sym}$ is symmetric, with $\a_{i}\in S(\a_{i})_{k_{i}}$ and $\a_{i-1}=\a_{i}^{(k_{i})}$
for some $k_{i}$. 
Moreover, $\a_{i}$ is regular for
all $i=1, \cdots, r$

\end{lemma}

\begin{proof}
Recall that  analogues to  \cite[Proposition 6.10]{Deng23}  we know that every regular  multisegment
$\a$ can be obtained as 
\[
 \a=\a_{0}, \a_{1}, \cdots, \a_{r},
\]
where $\a_{r}$ is symmetric, with $\a_{i}\in S(\a_{i})_{k_{i}}$ and $\a_{i-1}=\a_{i}^{(k_{i})}$
for some $k_{i}$.
\end{proof}

\subsection{The parabolic KL polynomials and their interpretations by geometry of graded nilpotent classes}

For fixed $n \in \N$ and a pair of elements in $S_{n}$,  we can associate  
a Kazhdan Lusztig Polynomial $P_{x,y}(q)$. We know also that the coefficients
of such a polynomial are given by the dimensions of the intersection cohomology
of corresponding Schubert varieties in $GL_{n}/B$. 

Similar construction can give rise to a polynomial related to the Poincar\'e series 
of the intersection cohomology of the Schubert varieties in $GL_{n}/P$, where 
$P$ is a standard parabolic subgroup. This has been done in 
Deodhar \cite{D2} for general Coxeter System $(W, S)$. As indicated in 
the same article, when $G=GL_{n}$, due to the existence of
the fibration $G/P\rightarrow G/B$, everything reduces to 
the Borel case.

In this section, 
for certain multisegment $\a$, we shall relate the set $S(\a)$ to 
the orbits in $GL_{n}/P$, where the multiplicities appear to be the corresponding
Parabolic Kazhdan Lusztig Polynomials evaluated at $q=1$.

\begin{notation}
Let $S=\{\sigma_{i}: i=1, \cdots, n-1\}$ be a set of generators for $S_{n}$. 
For $J\subseteq S$, let $S_{J}=<J>$ be the subgroup generated by $J$
and $S_n^{J}=\{w\in S_{n}: ws>w \text{ for all } s\in J\}$. 
\end{notation}

\begin{prop}(cf. \cite{BF} Prop. 2.4.4)\label{prop: 5.2.2}
 We have 
 \begin{description}
  \item [(1)]\(S_{n}=\coprod_{w\in S_n^{J}}wS_{J};\)
  \item [(2)]for $w\in S_n^{J}$ , and $x\in S_{J}$, $\ell(wx)=\ell(w)+\ell(x)$ .
 \end{description}

\end{prop}
\remk Now we can identify $S_n^{J}$ with $S_{n}/S_{J}$, hence it is 
in bijection with the Borel orbits in $GL_{n}/P$, where $P$ is the 
parabolic subgroup determined 
by $J$.

\begin{notation}
 Let $\a_{\Id}^J=\{\Delta_{1}, \cdots, \Delta_{n}\}$ such that 
 \[
  e(\Delta_{1})<\cdots< e(\Delta_{n}), 
 \]
and 
\[
  b(\Delta_{1})\leq \cdots\leq b(\Delta_{n}),
\]
such that 
\[
 b(\Delta_{i})=b(\Delta_{i+1}) \text{ if and only if } \sigma_{i}\in J
\]
and $b(\Delta_{n})\leq e(\Delta_{1})$.
\end{notation}
\begin{example}\label{ex: 6.2.4}
Let $n=4$, and  $J=\{\sigma_1, \sigma_3\}$, then we can choose  
\[
 \a_{\Id}^{J}=[1, 3]+[1, 4]+[2, 5]+[2, 6].
\] 
\end{example}

\begin{definition}\label{def: 6.2.5}
We call a multisegment $\a\in S(\a_{\Id}^J)$ a multisegment of parabolic type $J$.
\end{definition}

\begin{prop}\label{prop: 6.2.5}
 For $w\in S_n^{J}$, let $\a^J_{w}=\sum [b(\Delta_{i}), e(\Delta_{w(i)}) ]$, then $\a_{w}^J\in S(\a_{\Id}^J)$.
\end{prop}

\begin{example}
Let $\a_{\Id}^J$ as in Example \ref{ex: 6.2.4}. For $w=\sigma_1\sigma_2$, then 
\[
 \a_w^J=[1, 4]+[1, 5]+[2, 3]+[2, 6].
\]

\end{example}

\begin{proof}[of Proposition \ref{prop: 6.2.5}]
 We proceed by induction on $|J|$. If $|J|=0$, we are 
 in the symmetric case, so we are done by \cite[Proposition 4.7]{Deng23}.
 And in general, let $J=J_{1}\cup \{\sigma_{i_{0}}\}$ with 
 $i_{0}= \min\{i: \sigma_{i}\in J\}$ and 
 $i_{1}=\max\{i: b(\Delta_{i})=b(\Delta_{i_{0}})\}$.
 
 Let $\a_{1}= \{\Delta_{1}^{1}, \cdots, \Delta_{n}^{1}\}$, such that 
\begin{align*}
 \Delta_{i}^{1}&=^{+}(\Delta_{i}), \text{ for } i\leq i_{0}, \\
 \Delta_{i}^{1}&=\Delta_{i}, \text{ otherwise. }( \text{ cf. \cite[Definition 5.1]{Deng23}}) .
\end{align*}
See also the Example \ref{exm-construct-a1} after the proof.
Let $\a_{\Id}^{J_{1}}=\a_{1}$ with
\begin{displaymath}
 b(\Delta_i^1)=
 \left\{\begin{array}{cc}
 b(\Delta_i)-1, &\text{ for }i\leq i_{0},\\
b(\Delta_i), &\text{ for }i> i_0.
\end{array}\right.
\end{displaymath}

Then we  have
\[
 \a_{\Id}^J={^{(b(\Delta_{1}^1), \cdots, b(\Delta_{i_{0}}^1))}\a_{1}}.
\] 
Let $w_1=(i_1, \cdots, i_0+1, i_0)$, then $w_1\in S_n^{J_1}$.
Note that we have also $ww_{1}\in S_n^{J_{1}}$, since 
\[
 ww_{1}(i)=w(i-1)<ww_{1}(i+1)=w(i), \text{ for }i=i_{0}+1, \cdots, i_{1}-1.
\]

Then by induction, we know that 
\[
 \a_{ww_1}^{J_1}=\sum_{i}[b(\Delta_{i}^1), e(\Delta_{ww_1(i)}^1)]\in S(\a_{1}).
\]
An example (cf. Example \ref{exm-construction-a-Id-J}) is given after the proof.
Moreover,
\[
 \a_{w}^{J}={^{(b(\Delta_{1}^1), \cdots, b(\Delta_{i_{0}}^1))}\a_{ww_1}^{J_{1}}}.
\]
The result, that is the fact $\a_w^J\in S(\a_{\Id}^J)$ follows from the next lemma.
\end{proof}

\begin{example}\label{exm-construct-a1}
Let $\a_{\Id}^J$ be a multisegment as in Example \ref{ex: 6.2.4}. 
Then 
\[
 \a_1=[0, 3]+[1, 4]+[2, 5]+[2, 6].
\]
\end{example}

\begin{example}\label{exm-construction-a-Id-J}
 Let $\a_{\Id}^J$ as in the previous example. Then $i_1=2$, and $J_1=\{\sigma_3\}$.
 In this case, we have $w_1=\sigma_1$ and $ww_1=\sigma_1\sigma_2\sigma_1$, with 
 \[
  \a_{ww_1}^{J_1}=[0, 5]+[1, 4]+[2, 3]+[3, 6].
 \]
\end{example}

\begin{lemma}\label{lem: 6.2.7}
We have 
\[
 \a_{ww_1}^{J_{1}}\in {_{b(\Delta_{1}^1), \cdots, b(\Delta_{i_{0}}^1)}S(\a_{1})}.
\] 
\end{lemma}

\begin{proof}
In fact, let $\a_{1, 0}=\a_{\Id}^J$ and 
 for $j\leq i_{0}$, 
$\a_{1, j}= \{\Delta_{1, j}, \cdots, \Delta_{n,j}\}$, such that 
\begin{align*}
 \Delta_{i, j}&=^{+}(\Delta_{i}), \text{ for } i\leq j, \\
 \Delta_{i, j}&=\Delta_{i}, \text{ otherwise. }
\end{align*}
Then we have  
$\a_{1,j}={^{(b(\Delta_{j+1}^1), \cdots, b(\Delta_{i_{0}}^1))}\a_{1}}$,
for $j=0, 1, \cdots, i_{0}$.
For $ j<i_0-1$, let
\[
 \b_{j}=\sum_{j< i\leq i_{0}} [b(\Delta_{i}^{1})+1, e(\Delta_{w(i)}^{1}) ]+
 \sum_{i> i_{0},\text{ or } i\leq j} [b(\Delta_{i}^{1}), e(\Delta_{w(i)}^{1})],
\]
and $\b_{i_0}=\a_{ww_1}^{J_1}$
so that $\b_{j}={^{(b(\Delta_{j+1}^1), \cdots, b(\Delta_{i_{0}}^1))}\a_2}$.
We show that  $\b_{j}\in {_{b(\Delta_{j}^1)}S(\a_{1, j})}$ by induction on $j$. 
\begin{description}
 \item[(1)] For $j=i_{0}$,   we have 
 $$b(\Delta^{1}_{i_{0}})=b(\Delta^{1}_{i_{0}+1})-1=\cdots =b(\Delta^{1}_{i_{1}-1})-1=b(\Delta^1_{i_1})-1.$$
And $ww_1(i_0)>ww_1(i_1)>ww_{1}(i_1-1)>\cdots >ww_1(i_{0}+1)$, hence
 $$e(\Delta^{1}_{ww_1(i_{0})})>e(\Delta_{ww_1(i_1)}^1)>e(\Delta_{ww_1(i_1-1)}^1)>\cdots >e(\Delta_{ww_1(i_0+1)}^1),$$
 because $w\in S^{J}$.
This
implies  that $\b_{i_0}$ satisfies 
the hypothesis  $(_{b(\Delta^{1}_{i_{0}})}H(\a_{1,i_0}))$.
\item[(2)] For general $j\leq i_{0}-1$, 
By induction, we may assume that $\b_{j+1}\in {_{b(\Delta_{j+1}^1)}S(\a_{1, j+1})}$. 
Now to show  $\b_{j}\in {_{b(\Delta_{j}^1)}S(\a_{1, j})}$ , 
 we know that
$b(\Delta_{j}^{1})+1<b(\Delta_{j+1}^{1})+1$ in $\b_{j+1}$
(we have inequality by assumption on $i_{0}$), which proves that $\b_{j}\in {_{b(\Delta_{j}^1)}S(\a_{1, j})}$ .
 Hence we are done.
\end{description} 

\end{proof}

\begin{lemma}\label{lem: 6.2.8}
Let $J=\{\sigma_{i_{0}}\}\cup J_1$ such that 
$i_0=\min\{i: \sigma_{i}\in J\}$. Let
$i_1\in \Z$ be the maximal integer satisfying for $i_0\leq i<i_1$ 
we have $\sigma_{i}\in J$.
Then 
\[
 S_{J}^{J_1}=\{w_i: i=1, \cdots, i_1-i_0+1\}
\]
with 
\[
 w_i=(i_1-i+1, \cdots, i_0+1, i_0)\in S_J.
\]
As a consequence, we have 
\[
 S_n^{J_1}=\coprod_{i}S_n^Jw_i.
\]

\end{lemma}

\begin{proof}
By Proposition \ref{prop: 5.2.2}, we only need to show that
$S_J=\coprod_{j}w_jS_{J_1}$ and $w_j\in S^{J_1}$. 
The fact that $w_j\in S^{J_1}$ follows from 
\[
 w_j(i)=i-1, \text{ for }i=i_0+1, \cdots, i_1-j+1 , ~w_j(i_0)=i_1-j+1,
\]
and $w_j(i)=i$ for $i\notin \{i_0, \cdots, i_1-j+1\}$.
Finally, to see that $S_J=\coprod_{j}w_jS_{J_1}$, we compare the cadinalities.
Let $J_0=\{\sigma_{i}: i=i_0 \cdots, i_1-1 \}$, then
\[
 S_J\simeq S_{J_0}\times S_{J\setminus J_0},
\]
\[
 S_{J_1}\simeq S_{J_0\setminus \{\sigma_{i_0, i_0+1}\}} \times S_{J\setminus J_0}.
\]
Hence $\sharp{S_J}/\sharp{S_{J_1}}= \frac{\sharp{S_{J_0}}}{\sharp{S_{J_0\setminus \{\sigma_{i_0, i_0+1}\}}}}
=(i_1-i_0+1)!/(i_1-i_0)!=i_1-i_0+1$.
Finally, 
by Proposition \ref{prop: 5.2.2}, we know that 
\[
 S_n=\coprod_{v\in S_n^J}vS_J=\coprod_{j=i_0}^{i_1-i_0+1}\coprod_{v\in S_n^J}vw_{j}S_{J_1}=\coprod_{j}S_n^Jw_jS_{J_1}.
\]

\end{proof}

\begin{lemma}\label{lem: 6.2.9}
We keep the notations from Proposition \ref{prop: 6.2.5}. For $i=1, \cdots, i_1-i_0+1$, we have 
\[
 \a_w^J={^{(b(\Delta_{1}^1), \cdots, b(\Delta_{i_{0}}^1))}\a_{ww_i}^{J_{1}}}.
\] 
\end{lemma}

\begin{proof}
Note that by definition
we have 
\[
 \a_{ww_j}^{J_1}=\sum_{i}b(\Delta_{i}^1), e(\Delta_{ww_j(i)}^1)].
\]
As noted before, we have 
\[
 b(\Delta_i^1)=b(\Delta_i)-1, \text{ for }i\leq i_{0}, ~b(\Delta_i^1)=b(\Delta_i), \text{ for }i> i_0.
\]
Also,we observe that 
$e(\Delta_{i}^1)=e(\Delta_i)$.
Hence 
\[
 {^{(b(\Delta_{1}^1), \cdots, b(\Delta_{i_{0}}^1))}\a_{ww_i}^{J_{1}}}=
 \sum_{i}b(\Delta_{i}), e(\Delta_{ww_j(i)})].
\]
It remains to see that we have 
\[
 \sum_{i=i_0}^{i_1-j+1}b(\Delta_{i}), e(\Delta_{ww_j(i)})]=\sum_{i_0}^{i_1-j+1}b(\Delta_{i}), e(\Delta_{w(i)})]
\]
since $b(\Delta_{i_0})=\cdots=b(\Delta_{i_1-j+1})$.
Hence we have 
\[
  \a_w^J={^{(b(\Delta_{1}^1), \cdots, b(\Delta_{i_{0}}^1))}\a_{ww_i}^{J_{1}}}.
\] 
\end{proof}

\begin{definition}
As in the symmetric cases, we have the following map 
\begin{align*}
 \Phi_{J}: S_n^{J}&\rightarrow S(\a_{\Id}^J)\\
 w&\mapsto \a_{w}^J.
\end{align*}
\end{definition}

\begin{prop}\label{prop: 6.2.10}
The morphism $\Phi_{J}$ is bijective and translate the inverse Bruhat order
on $S_n^{J}$ to the order on $S(\a_{\Id}^J)$.
\end{prop}

\begin{proof}
Again, we do this by induction on $|J|$. If $|J|=0$, 
we are in the symmetric case, so everything is in \cite[Proposition 4.7]{Deng23}.
In general, we keep the notations in the Proposition \ref{prop: 6.2.5}.
We have $J=J_{1}\cup \{\sigma_{i_{0}}\}$.
By Lemma \ref{lem: 6.2.7}, 
\[
 \a_{ww_1}^{J_{1}}\in {_{b(\Delta_{1}^1), \cdots, b(\Delta_{i_{0}}^1)}S(\a_{1})}.
\]

By Lemma  \ref{lem: 6.2.9} and Proposition \ref{prop: 6.2.5} the morphism 
${_{b(\Delta_{1}^1), \cdots, b(\Delta_{i_{0}}^1)}\psi}$ (cf. \cite[Definition 6.6]{Deng23})
sends $\Phi_{J_{1}}(ww_1)$ to 
$\Phi_{J}(w)$ for $w\in J$. 
Therefore 
\[
 \Phi_J={_{b(\Delta_{1}^1), \cdots, b(\Delta_{i_{0}}^1)}\psi}\circ \Phi_{J_1},
\]
and the injectivity of follows from that of ${_{b(\Delta_{1}^1), \cdots, b(\Delta_{i_{0}}^1)}\psi}$
and induction on $J_1$.
For surjectivity,
let $\b\in S(\a_{\Id}^J)$. By the analogues of \cite[Proposition 6.5]{Deng23}, we have the surjectivity of the map 
\[
 {_{b(\Delta_{1}^1), \cdots, b(\Delta_{i_{0}}^1)}\psi}: 
 {_{b(\Delta_{1}^1), \cdots, b(\Delta_{i_{0}}^1)}S(\a_{1})}\rightarrow S(\a_{\Id}^J).
\]
Hence we know that there exists a $w'\in S^{J_1}$, such that 
$\Phi_{J_1}(w')\in  {_{b(\Delta_{1}^1), \cdots, b(\Delta_{i_{0}}^1)}S(\a_{1})}$, and 
is sent to $\b$ by $ {_{b(\Delta_{1}^1), \cdots, b(\Delta_{i_{0}}^1)}\psi}$.
By lemma \ref{lem: 6.2.8},  every $w'\in S^{J_{1}}$ can be write as 
$ww_j$ for some $w\in S^J$ and $w_j\in S^{J_1}$. Now by lemma \ref{lem: 6.2.9},
\[
 \b=\a_{w}^J.
\]
Note that for $w>w'$ in $S^J$, then $ww_1>w'w_1$ in $S^{J_1}$, hence by induction
\[
 \Phi_{J_1}(ww_1)<\Phi_{J_1}(ww_1),
\]
we get 
\[
 \Phi_{J_1}(w)<\Phi_{J_1}(w),
\]
since the morphism ${_{b(\Delta_{1}^1), \cdots, b(\Delta_{i_{0}}^1)}\psi}$
preserves the order.

\end{proof}

\begin{teo}\label{teo-main-regular}
Let $v_{1}, v_{2}\in S^{J}$, then we have 
\[
 P_{\Phi_{J}(v_{1}), \Phi_{J}(v_{2})}(q)=P_{v_{1}, v_{2}}^{J}(q)
\]
where on the right hand side is the parabolic KL polynomial indexed by $v_{1}, v_{2}$.
\end{teo}

\begin{proof}
As is proved in \cite{D2}, we have $P_{v_{1}, v_{2}}^{J}(q)=P_{v_{1}v_{J}, w_{2}v_{J}}(q)$, 
where $v_{J}$ is the maximal element in $S_{J}$. So it suffices to show that we have 
the equality $P_{\Phi_{J}(v_{1}), \Phi_{J}(v_{2})}(q)=P_{v_{1}v_{J}, v_{2}v_{J}}(q)$.
Also, from Lemma \ref{lem: 6.2.7}, we know that 
\[
 \Phi_{J}(v_1)= {_{b(\Delta_{1}^1), \cdots, b(\Delta_{i_{0}}^1)}\psi}(\Phi_{J_1}(v_1w_1)),
\]
where $w_1$ is described in Lemma \ref{lem: 6.2.8}.
Hence we have 
\[
 P_{\Phi_{J_{1}}(v_1w_1), \Phi_{J_{1}}(v_{2}w_1)}(q)=
 P_{\Phi_{J}(v_{1}), \Phi_{J}(v_{2})}(q)
\]
by \cite[Corollary 5.37]{Deng23}.

By induction, we have 
\[
 P_{\Phi_{J_{1}}(v_{1}w_1), \Phi_{J_{1}}(v_{2}w_1)}(q)
 =P_{v_{1}w_1v_{J_{1}}, v_{1}w_1v_{J_{1}}}(q).
\]
Now to finish, we have to show $v_{J}=w_1v_{J_{1}}$.
But we know that 
\[
 S_J=\coprod_{j}w_jS_{J_1}
\]
with $w_1=\max\{w_j: j=1, \cdots, i_1-i_0+1\}$, we surely have
\[
 v_{J}=w_1v_{J_{1}}.
\]
\end{proof}

More generally, for $J_{i}\subseteq S$, $i=1,2$, we can consider
the $P_{J_{1}}$ orbit in $GL_{n}/P_{J_{2}}$. We state the related results
without proof since it is exactly the same as the previous treated case. 

\begin{definition}
Let $S_n^{J_{1}, J_{2}}=\{w\in S_{n}: s_{1}vs_{2}>v \text{ for all } s_{i}\in J_{i}, i=1,2\}$. 
\end{definition}

\begin{definition}\label{def: 6.2.17}
Let $v\in S_n^{J_1, J_2}$. We define 
\[
 S_{J_1}^{J_2, v}=\{w\in S_{J_1}: ws>w, \text{ for all } s\in S_{J_1}\cap vS_{J_2}v^{-1}\}.
\]
 
\end{definition}

\remk If we let $M_J$ be the Levi subgroup of $P_J$, then the set 
$S_{J_1}^{J_2, v}$ corresponds to the Borel orbits in $M_{J_1}/(M_{J_1}\cap vM_{J_2}v^{-1})$.

\begin{prop}\label{prop: 6.2.13}
 We have 
 \begin{description}
  \item [(1)]$S_{n}= \coprod_{v\in S_n^{J_{1}, J_{2}}}S_{J_{1}}vS_{J_{2}}$;
  \item [(2)]$\ell(xvy)=\ell(v)+\ell(x)+\ell(y)$ for $v\in S^{J_{1}, J_{2}}$ , $x\in S_{J_{1}}^{J_2, v}, y\in S_{J_{2}}$.
  \item [(3)]The $P_{J_{1}}$ orbits in $GL_{n}/P_{J_{2}}$ are indexed 
   by $S_n^{J_{1}, J_{2}}$.
 \end{description}  
\end{prop}

\begin{definition}
 For $v_{1}, v_{2}\in S^{J_{1}, J_{2}}$ such that $v_{1}\leq v_{2}$, we let $P^{J_{1}, J_{2}}_{v_{1}, v_{2}}(q)$ be the 
 Poincar\'e series of  the localized intersection cohomology 
 $$\mathcal{H}^{\bullet}(\line {P_{J_{1}}v_{2}P_{J_{2}}})_{v_{1}P_{J_{2}}}.$$
\end{definition}

\begin{lemma}
For $v_{1}, v_{2}\in S^{J_{1}, J_{2}}$ such that $v_{1}\leq v_{2}$,
 we have 
 \[
  P^{J_{1}, J_{2}}_{v_{1}, v_{2}}(q)=P_{w_1, w_2}(q),
 \]
where $w_i$ is the element of maximal length in $S_{J_{1}}v_iS_{J_2}$.
\end{lemma}

\begin{notation}\label{nota: 6.2.21}
 Let $\a_{\Id}^{J_1, J_2}=\{\Delta_{1}, \cdots, \Delta_{n}\}$ such that 
 \[
  e(\Delta_{1})\leq \cdots\leq  e(\Delta_{n}), 
 \]
 such that 
\[
 e(\Delta_{i})=e(\Delta_{i+1}) \text{ if and only if } \sigma_{i}\in J_{1}
\]
and 
\[
  b(\Delta_{1})\leq \cdots\leq b(\Delta_{n}),
\]
such that 
\[
 b(\Delta_{i})=b(\Delta_{i+1}) \text{ if and only if } \sigma_{i}\in J_{2}
\]
and $b(\Delta_{n})\leq e(\Delta_{1})$.
\end{notation}

\begin{definition}\label{def: 6.2.22}
We call a multisegment $\a\in S(\a_{\Id}^{J_1, J_2})$ a multisegment of parabolic type $(J_{1}, J_{2})$.
\end{definition}

\begin{lemma}\label{lem: 6.2.3}
 For $w\in S^{J_{1}, J_{2}}$, let $\a_{w}^{J_1, J_2}=\sum [b(\Delta_{i}), e(\Delta_{w(i)}) ]$, then 
 $\a_{w}^{J_1, J_2}\in S(\a_{\Id}^{J_1, J_2})$.
Therefore we have an application
\begin{align*}
 \Phi_{J_{1}, J_{2}}: S^{J_{1}, J_{2}}&\rightarrow S(\a_{\Id}^{J_1, J_2})\\
 w&\mapsto \a_{w}^{J_1, J_2}.
\end{align*}
\end{lemma}

\begin{prop}\label{prop: 6.2.4}
The morphism $\Phi_{J_{1}, J_{2}}$ is bijective and translate the inverse Bruhat order
on $S^{J_{1}, J_{2}}$ to the order on $S(\a_{\Id}^{J_1, J_2})$.
\end{prop}

\begin{prop}\label{prop: 6.2.25}
Let $w_{1}, w_{2}\in S_n^{J_{1},J_{2}}$, then we have 
\[
 P_{\Phi_{J_{1}, J_{2}}(w_{1}), \Phi_{J_{1}, J_{2}}(w_{2})}(q)=P_{w_{1}, w_{2}}^{J_{1}, J_{2}}(q)
\]
where on the right hand side is the parabolic KL polynomial indexed by $w_{1}, w_{2}$.
\end{prop}

\begin{example}
 We are now ready to interpret the following results
(due to Zelevinsky, see \cite{Z3} Section 3.3): let $\a=k[0, 1]+(n-k)[1,2]$
 then $\a$ corresponding to the identity in $S_n^{J, J}$ with 
 \[
  J=\{\sigma_{i}: i\neq k\}. 
 \]
Note that in this case, we have $GL_{n}/P_{J}$ is the 
Grassmanian $G_{k}(\C^{n})$, where as the $P_{J}$
orbits correspond to the stratification, for $r\leq r_{0}=\min\{k, n-k\}$
and fixed $\C^k\in G_{k}(\C^n)$, 
\[
 X_{r}=\{U\in G_{k}(\C^{n}): \dim(U\cap \C^{k})= k-r\}
\]
with $\line{X_{r}}=\coprod_{r'\leq r}X_{r'}$.
\end{example}

\remk Note that in \cite{Z3}, Zelevinsky constructed a small resolution 
for the $\line{O}_{\a}$ with $\a=\{[1, 2], [2, 3]\}$, which corresponds
to a Schubert varieties of 2-step . Now with our interpretation, we should 
be able to construct a small resolution for all 2-step Schubert varieties.
We return to this question later.

\subsection{Non regular Case}

In this subsection, for a general multisegment $\a$, we will relate the poset $S(\a)$ 
to a Bruhat interval $[x,y]$ with $x<y$ in some $S_{r}^{J_1, J_2}$. 

Now let $\a$ be a multisegment. 
First of all, we decide 
the set $J_1, J_2$. 

\begin{definition}\label{def: 6.3.1}
We define two sets $J_1(\a), J_2(\a)$.
\begin{itemize}
 \item 
Let $b(\a)=\{k_1\leq \cdots\leq k_r\}$.
Then let $J_2(\a)\subseteq S_r$ be the set such that
 $\sigma_{i}\in J_2(\a)$ if and only if $k_i=k_{i+1}$. 
\item Let $e(\a)=\{\ell_1\leq \cdots \leq \ell_r\}$.
Then let $J_1(\a)\subseteq S_r$ be the set such that $\sigma_{i}\in J_1(\a)$ if and only if $\ell_{i}=\ell_{i+1}$.
\end{itemize}
\end{definition}


Keeping the notations in Definition \ref{def: 6.3.1},
\begin{prop}\label{prop: 6.3.2}
There exists a unique $w\in S_r^{J_1(\a), J_2(\a)}$, such that 
\[
 \a=\sum_{j}[k_j, \ell_{w(j)}].
\]
\end{prop}

\begin{proof}
We observe that there exists an element $w'\in S_r$, such that 
\[
\a=\sum_j [k_j, \ell_{w'}(j)]. 
\]
Now by Proposition \ref{prop: 6.2.13},
we know that there exists $w'=w_{J_1(\a)}ww_{J_2(\a)}$
with $w_{J_i(\a)}\in S_{J_i(\a)}$ for $i=1, 2$ and $w\in S_r^{J_1(\a), J_2(\a)}$.
Now we only need to prove that 
\[
 \a=\sum_j [k_j, \ell_{w}(j)].
\]
In fact, by definition of $J_i(\a), i=1,2$, we know that
\[
 k_j=k_{v(j)},  \text{ for all } v\in S_{J_{2}(\a)},
\]
\[
 \ell_{j}=\ell_{v(j)}, \text{ for all } v\in S_{J_{1}(\a)}.
\]
Hence 
\begin{align*}
 \a&=\sum_j [k_j, \ell_{w_{J_1(\a)}ww_{J_{2}(\a)}(j)}]\\
   &=\sum_j [k_j, \ell_{w(\a)w_{J_2(\a)}(j)}]\\
   &=\sum_j [k_{w_{J_2(\a)}^{-1}w^{-1}(j)}, \ell_{j}]\\
   &=\sum_j [k_j, \ell_{w(j)}].
\end{align*}
\end{proof}
 
Next we show how to reduce a general multisegment $\a$
to a multisegment $\a_{w}^{J_{1}(\a), J_2(\a)}$ of parabolic 
type $(J_1(\a), J_2(\a))$  without changing the poset structure $S(\a)$.

\begin{prop}\label{prop: 6.3.3}
Let $\a$ be a multisegment, then there exists a multisegment  $\c$, 
and a multisegment $\a_{w}^{J_{1}(\a), J_2(\a)}$  of parabolic type $(J_1(\a), J_2(\a))$, such that 
\[
 \a_{w}^{J_{1}(\a), J_2(\a)}\in S(\a_{w}^{J_{1}(\a), J_2(\a)})_{\c}, ~\a=(\a_{w}^{J_{1}(\a), J_2(\a)})^{(\c)}
\]

\end{prop}
\begin{proof}
 
 In general $\a$ is not of parabolic type, i.e,
we do not have $\min\{e(\Delta): \Delta\in \a\}\geq \max\{b(\Delta): \Delta\in \a\}$.
Now we show how to construct  $\a_{w}^{J_{1}(\a), J_2(\a)}$ .
 
 In fact, let 
\[
 \a=\{\Delta_{1}, \cdots, \Delta_{n}\},  \Delta_{1}\prec \cdots \prec \Delta_{n}.
\]
Then
\[
 e(\Delta_{1})=\min\{k: i=1, \cdots, n\}.
\]
If $\a$ is not of parabolic type, let
$\Delta^{1}=[e(\Delta_{1})+1, \ell] $ with $\ell$ maximal satisfying 
that for any $m$ such that $ e(\Delta_{1})\leq m\leq  \ell-1$, there is a segment in $\a$ ending in $m$.
We construct $\a^{1}$ by replacing every segment $\Delta$ in 
$\a$ ending in $\Delta^{1}$ by $\Delta^{+}$.
Repeat this construction with $\b_{1}$ to get $\a^{2}\cdots $,
until we get $\a^{s}$, which is of parabolic type.
Let $\c=\{\Delta^{1}, \cdots, \Delta^{s}\}$, then we do as in  \cite[Proposition 6.9]{Deng23} to get 
\[
 \a^s \in S(\a^s)_{\c},~ \a=(\a^s)^{(\c)}.
\] 
Note that by our construction we have 
\[
 J_1(\a^i)=J_1(\a),~ J_2(\a^i)=J_2(\a),
\]
for $i=1, \cdots, s$. 
\end{proof}

\begin{lemma}
Assume that $\a\in S(\a)_{k}$ such that 
\[
 J_1(\a)=J_1(\a^{(k)}), ~J_2(\a)=J_2(\a^{(k)}).
\]
Let $\c\in S(\b)_k$. Then for $\d\in S(\b)$ and $\d>\c$, we have $\d\in S(\b)_k$. 
\end{lemma}

\begin{proof}
It suffices to show that $\d$ satisfies the hypothesis $H_{k}(\a)$.
Note that $e(\d)\subseteq e(\a)$ as is indicated in \cite[Proposition 4.4]{Deng23}.
Assume that $k\in e(\a)$ to avoid triviality.
Now that $k-1\notin e(\a)$ since $\a\in S(\a)_k$ and 
\[
 J_1(\a)=J_1(\a^{(k)}), ~J_2(\a)=J_2(\a^{(k)}),                                                  
\]
 so it is also not in $e(\d)$. Hence to show that $\d\in S(\b)_k$
 hence it is equivalent to show that $\varphi_{e(\d)}(k)=e_{e(\a)}(k)$.
Since $\c\in S(\d)$, we know that $e(\c)\subseteq e(\d)$ hence 
$\varphi_{e(\d)}\leq \varphi_{e(\d)}(k)$.
Now that 
$\c\in S(\a)_k$ implies 
$\varphi_{e(\c)}=\varphi_{e(\a)}$, we conclude that $\varphi_{e(\d)}(k)=e_{e(\a)}(k)$. 
We are done. 
\end{proof}

Now let $\a'\in S(\a)_k$ such that $\psi_k(\a')=(\a^{(k)})_{\min}$, then 
\begin{lemma}\label{lem: 6.3.5}
We have 
\[
 S(\a)_k=\{\c\in S(\a): \c\geq\a'\}.
\] 
\end{lemma}

\begin{proof}
 By the lemma above, we know that 
 $S(\a)_k\supseteq \{\c\in S(\a): \c\geq\a'\}$.
 We conclude that we have equality since $\psi$ preserve the order.
\end{proof}

\begin{teo}\label{prop: 6.3.6}
Assume that $\a$ is a multisegment. Then 
\begin{description}
\item[(1)] 
There exists a multisegment 
$\a_{w}^{J_{1}(\a), J_2(\a)}$ of parabolic type $(J_1(\a), J_2(\a))$ and 
a sequence of integers $k_1, \cdots, k_r$
such that 
\[
 S(\a)\simeq S(\a_{w}^{J_{1}(\a), J_2(\a)})_{k_{r}, \cdots, k_{1}}.
\]

\item[(2)] There exists an element $\a'\in S(\a_{w}^{J_{1}(\a), J_2(\a)})$ such that 
\[
 S(\a_{w}^{J_{1}(\a), J_2(\a)})_{k_{r}, \cdots, k_{1}}=\{\c\in S(\a_{w}^{J_{1}(\a), J_2(\a)}): \c\geq \a'\}.
\]

\end{description}
\end{teo}
 
\begin{proof}
Note that (1) follows from  \cite[Proposition 6.5]{Deng23} and 
Proposition \ref{prop: 6.3.3}.
And (2) follows from 
applying the Lemma \ref{lem: 6.3.5} successively to the lemma below. 
\end{proof}
 
\begin{lemma}
There exists a sequence of multisegments  $\a_{0}=\a, \cdots,\a_{r}=\a_{w}^{J_{1}(\a), J_2(\a)}$ such that 
$\a_{w}^{J_{1}(\a), J_2(\a)}$ is of parabolic type $(J_1(\a), J_2(\a))$,   $\a_{i}\in S(\a_{i})_{k_{i}}$ and $\a_{i-1}=\a_{i}^{(k_{i})}$
for some $k_{i}$. 
Moreover, 
\[
 J_{1}(\a_i)=J_1(\a), ~J_2(\a)=J_2(\a)
\]
for
all $i=1, \cdots, r$.  
\end{lemma}

\begin{proof}
 This follows from our construction in the 
 proof of Proposition \ref{prop: 6.3.3}.
\end{proof}

\bibliographystyle{plain}
\bibliography{biblio}

\end{document}